\documentclass[12pt]{amsart}
\usepackage{amscd,graphicx,epsfig}
\usepackage[colorlinks=true, pdfstartview=FitV, linkcolor=blue, citecolor=blue, urlcolor=blue]{hyperref}
\usepackage{amssymb}
\usepackage{xcolor}

\usepackage[english]{babel}
\usepackage[utf8]{inputenc}

\title[Gromov ellipticity and subellipticity]{Gromov ellipticity and subellipticity} 
\author{Shulim Kaliman and Mikhail Zaidenberg}

\address{University of Miami, Department of Mathematics, Coral Gables, FL
33124, USA}
\email{kaliman@math.miami.edu}
\address{Univ. Grenoble Alpes, CNRS, IF, 38000 Grenoble, France}
\email{mikhail.zaidenberg@univ-grenoble-alpes.fr}

\date{06.01.2023}

\newtheorem{thm}{Theorem}[section]
\newtheorem*{thm*}{Theorem}
\newtheorem*{conj*}{Conjecture}

\newtheorem{prop}[thm]{Proposition}
\newtheorem{lem}[thm]{Lemma}
\newtheorem{cor}[thm]{Corollary}
\newtheorem*{cor*}{Corollary}

\theoremstyle{definition}

\newtheorem*{exa*}{Example}

\theoremstyle{remark}
\newtheorem*{rem*}{Remark}

\theoremstyle{remark}
\newtheorem*{rems*}{Remarks}

\newcommand{\NN}{{\mathbb N}}

\newcommand{\kk}{{\mathbb K}}
\newcommand{\GG}{{\mathbb G}}

\renewcommand{\AA}{{\mathbb A}}

\newcommand{\OOO}{\mathcal O}

\newcommand{\ee}{\end{enumerate}}

\DeclareMathOperator{\Pic}{Pic}

\DeclareMathOperator{\GL}{GL}

\renewcommand{\subset}{\subseteq}

\renewcommand{\phi}{\varphi}

\begin{document}
{\small
\begin{abstract}
We establish the equivalence of Gromov ellipticity and subellipticity in the algebraic category. 
\end{abstract}
}
\maketitle
{\small
\tableofcontents}
\bigskip
\section*{Introduction}
\subsection{Ellipticity versus subellipticity}
We are working over an algebraically closed field $\kk$ of characteristic zero; 
and $\AA^n=\AA^n_\kk$ stands for the
affine $n$-space over $\kk$. All varieties and vector bundles in this note are algebraic; a variety is always reduced and irreducible. 
Abusing the language, in this paper the terms \emph{spray} and \emph{(sub)ellipticity} 
stand for algebraic spray resp. algebraic (sub)ellipticity. 

The notion of ellipticity was introduced, both in analytic and algebraic setting, in \cite[Sect. 0.5 and 3.5.A]{Gro89}. 
In the Localization Lemma 
(see \cite[Sect. 3.5.B]{Gro89}) Gromov actually considers a local version of ellipticity. 
The subellipticity (implicitly present in \cite[Sec. 3.5]{Gro89}) was introduced in
\cite{For02}; see \cite[Sect. 5.1]{For17} and \cite{For20} for a historical account. Recall these notions. 

A \emph{spray of rank $r$} over a smooth algebraic variety $X$ is a triple $(E,p,s)$ consisting of a vector bundle 
$p\colon E\to X$ of rank $r$ and a morphism $s\colon E\to X$ such that $s|_Z=p|_Z$ where $Z\subset E$ stands 
for the zero section of $p$. This spray is \emph{dominating at $x\in X$} if the restriction $s|_{E_x}\colon E_x\to X$ 
to the fiber $E_x=p^{-1}(x)$ is dominant at the origin  $0_x= Z\cap E_x$
of the vector space $E_x$. 
The variety $X$ is called \emph{elliptic} if it admits a spray $(E,p,s)$ which is dominating at each point $x\in X$.

A \emph{local spray} $(E,p,s)$ with values in $X$ at a point $x\in X$ consists of a  vector bundle 
$p\colon E\to U$ over a neighborhood $U$ of $x$ and a morphism $s\colon E\to X$ such that 
$s|_{Z_U}=p|_{Z_U}$ where $Z_U$ stands for the zero section of $p\colon E\to U$.
One says that $X$ is \emph{locally elliptic} if for any $x\in X$ there is a local spray in 
a neighborhood of $x$ dominating at $x$. The variety $X$ is called 
\emph{subelliptic} if it
admits a family of sprays $(E_i,p_i,s_i)$ defined over the whole $X$ which is dominating at each point $x\in X$, that is,
\[T_xX=\sum_{i=1}^n {\rm d}s_i (T_{0_{i,x}} E_{i,x}) \quad \forall x\in X.\]
Thus, an elliptic variety is both locally elliptic and subelliptic. 
In this note we establish the converse implications.
\begin{thm}\label{mthm}
For a smooth algebraic variety $X$ the following are equivalent:
\begin{itemize} 
\item $X$ is elliptic;
\item $X$ is subelliptic;
\item $X$ is locally elliptic.
\end{itemize}
\end{thm}
Theorem \ref{mthm} answers a question in \cite[p.~230]{For17}. 
The equivalence of the first two properties is well known 
for homogeneous spaces of a linear algebraic group, see \cite[Proposition 6.7]{For20}. 
 It is unknown however whether this equivalence  also  holds in the analytic category; cf. \cite{Kus20}.

Comparing Theorem \ref{mthm} with \cite[Theorem 1]{LT17} and \cite[Corollary 6.6]{For20} we deduce the
following results, see \cite{LT17} and \cite{For20} for the
terminology. 
\begin{cor}\label{el.c2} 
For a smooth algebraic manifold $X$ the following conditions are equivalent:
\begin{enumerate}
\item[(a)] $X$ is elliptic;
\item[(b)] $X$ is locally elliptic;
\item[(c)] $X$ is subelliptic;
\item[(d)] $X$ satisfies Gromov's condition ${\rm aEll}_1$;
\item[(e)] $X$ satisfies the algebraic homotopy Runge principle.
\end{enumerate}
\end{cor}
See \cite[Sec. 6.2]{For20} for relations of (d) and (e) to approximation, interpolation and 
Oka-Grauert $h$-Principle in the algebraic setting.

As another immediate corollary we mention the following version of \cite[Corollary 1.5]{Kus22}, 
see also \cite[Corollary 6.26]{For20}. 
\begin{cor}\label{cor:kusa} 
The universal cover of a smooth (sub)elliptic variety is an elliptic algebraic variety.
\end{cor}
\subsection{Examples of elliptic varieties} 
Recall that a variety $X$ of dimension $\ge 2$ is \emph{flexible} if through any smooth point $x\in X$ 
pass one-dimensional orbits of $\GG_a$-actions on $X$ whose velocity vectors spend 
the tangent space $T_xX$, see \cite{AFKKZ13}.  It is known that every smooth flexible variety is elliptic, 
see \cite[0.5.B]{Gro89} and \cite[Proposition 5.6.22(C)]{For17}; cf. also \cite[Appendix]{AFKKZ13}.
There are numerous examples of flexible varieties; see e.g.  \cite{AFKKZ13} and survey articles 
\cite{CPPZ21} and \cite{Arz22}. 

 The complement of a closed subset $Y$ of codimension $\ge 2$
in a flexible smooth quasi-affine variety $X$ of dimension $\ge 2$ is again flexible, 
see \cite[Theorem 1.1]{FKZ16}.
In particular, $X\setminus Y$ is elliptic. 

Recall that an algebraic variety $X$ of dimension $n$ belongs to class $\mathcal{A}_0$ 
if it can be covered by a finite number of copies of $\AA^n$. It belongs to class $\mathcal{A}$
 if it is the complement of a closed subvariety of codimension at least $2$
in a variety of class  $\mathcal{A}_0$. Any variety of class $\mathcal{A}$ is subelliptic, 
see \cite[Proposition 6.4.5]{For17}.  By Theorem \ref{mthm} it is elliptic. 

For example, the Grassmann manifold $X={\mathbb G}(k,n)$ 
of $k$-dimensional subspaces in $\AA^n$ belongs to class $\mathcal{A}_0$. Hence $X\setminus Y$ is elliptic for any 
closed subset $Y\subset X$ of codimention $\ge 2$, see 
\cite[Corollary 5.6.18(D)]{For17}.
This answers a question in [\emph{ibid}]. 

Furthermore, a variety of class $\mathcal{A}$ blown up along a smooth closed subvariety is subelliptic. The same concerns a smooth locally stably flexible variety, see \cite{LT17}, \cite{KKT18} and \cite[Theorem 6.4.8]{For17}. By Theorem \ref{mthm} all these varieties are elliptic.
See also \cite[Sec. 3.4(F)]{Gro89} for further potential examples.  
\subsection{Ellipticity versus rationality}
By definition, any  smooth (sub)elliptic variety is dominated by an affine space, hence is unirational and rationally connected.
Notice, however, that there are examples of flexible and so, 
elliptic smooth affine varieties that are not stably rational, see \cite[Example 1.22]{Pop11}. Gromov asked in \cite{Gro89}
whether any rational projective variety is elliptic. The same question can be asked for  unirational projective varieties, 
see \cite{BKK13}. 
%
%
\section{Composing and extending sprays}\label{sec1}
In this section $X$ stands for a smooth algebraic variety. 
We develop here several technical tools for the proof of Theorem \ref{mthm}.
\subsection{Composition of sprays}\label{Composition}
The \emph{composition} $(E_1*E_2,p_1*p_2, s_1*s_2)$  of two given sprays
 $(E_1, p_1,s_1)$ and $(E_2, p_2,s_2)$ over $X$ is defined via\footnote{Abusing notation, we do not distinguish between a vector bundle and its total space.} 
\[ \begin{array}{c} E_1*E_2=\{ (e_1,e_2) \in E_1\times E_2\,|\, e_1 \in p^{-1} (X), \,\,\,s_1(e_1)=p_2(e_2)\}=s_1^*(E_2),
\\ p_1*p_2(e_1,e_2)=p_1(e_1), \qquad
s_1*s_2(e_1,e_2)=s_2(e_2),\end{array}\]
see \cite[1.3.B]{Gro89}.
One considers also the iterated composition 
\[(E,p,s)=(E_1*\ldots *E_m,\,\,\,p_1*\ldots *p_m,\,\,\, s_1*\ldots *s_m)\] 
of a sequence of $m$ sprays $(E_i, p_i,s_i)$ over $X$, see [\emph{ibid}] or \cite[Definition 6.3.5]{For17}. 
Clearly, $p\colon E\to X$ is a fiber bundle whose general fiber is isomorphic to an affine space $\AA^N$ viewed as a variety. 
In general, $p\colon E\to X$ does not admit a vector bundle structure. 
However, it admits a natural section 
$\sigma\colon X\to E$ (which plays a role of zero section)
such that  $s|_{\sigma(X)}=p|_{\sigma(X)}$, see \cite[Definition 6.3.5]{For17}. 

If $p_i\colon E_i\to X$ are trivial vector bundles for all $i$, then the composition $(E,p,s)$ 
is a spray over $X$ with a trivial vector bundle $p\colon E\to X$, see \cite[Lemma 6.3.1]{For17}. 
Furthermore, we have the following fact. 
\begin{prop}\label{el.p1}
Consider  the composition $(E,p,s)=(E_1*E_2,p_1*p_2, s_1*s_2)$  of
sprays  $(E_1,p_1, s_1)$ and $(E_2,p_2, s_2)$ over $X$. If $(E_2,p_2, s_2)$
is of rank 1, then $(E,p, s)$ is a spray over $X$.
\end{prop}

\begin{proof} 
By the preceding it suffices to show that $p\colon E\to X$ is a vector bundle. 
The pullback via the projection $p_1\colon E_1\to X$ yields an isomorphism $\Pic(X)\cong \Pic(E_1)$, 
see \cite[Theorem 5]{Mag75}.  
Therefore, the line bundle $s_1^*(E_2)\to E_1$ is isomorphic to $p_1^*(L_1)$ 
for a line bundle $L_1\to X$. 
 
Choose an affine open cover  $\{U_i\} $  of $X$ which is trivializing simultaneously for 
$p_1\colon E_1\to X$ and for $L_1\to X$
and consider the cylinders $Y_i:=p_1^{-1} (U_i)\simeq U_i\times \AA^r$ and 
$p_1^*(L|_{U_i})\simeq U_i\times \AA^r\times \AA^1$.
Up to these trivializations the composition $p_1^*(L_1|_{U_i})\to U_i$ coincides 
with the standard projection ${\rm pr}_1\colon U_i\times \AA^{r+1}\to U_i$. 
The transition function $\varphi_{i,j}$ between 
$p_1^*(L_1|_{U_i})\to U_i$ and $p_1^*(L_1|_{U_j})\to U_j$ over $U_i\cap U_j$
 is of the form 
 \[\varphi_{i,j}\colon (x, v, t) \mapsto (x, f_{i,j}(x,v), g_{i,j}(x,v)t), \quad (x, v, t)\in (U_i\cap U_j)\times \AA^r\times \AA^1\]
where 
\[f_{i,j}\in\GL(r, \OOO(U_i\cap U_j))\quad\text{and}\quad g_{i,j}\in \OOO^*(Y_i\cap Y_j)=p_1^*(\OOO^*(U_i\cap U_j)), \]
that is, the functions $g_{i,j}(x,v)=g_{i,j} (x)$ do not depend on $v\in\AA^r$.
Finally $\varphi_{i,j}\in \GL(r+1, \OOO(U_i\cap U_j))$,
 which proves that $p=p_1*p_2 \colon E \to X$ is a vector bundle.
\end{proof}

\begin{cor}\label{cor:comp-1-sprays}
The composition $(E,p,s)$ of $r$ rank $1$ sprays $(L_i, q_i,s_i)$ over $X$ is a spray of rank $r$ over $X$. 
If the family $(L_i, q_i,s_i)$ is dominating, then also the spray $(E,p,s)$ is.
\end{cor}
\begin{proof}
The first assertion is immediate from Proposition \ref{el.p1}; see \cite[Lemma 6.3.6]{For17} for the second.
\end{proof}
\subsection{Extension lemma}\label{ss:extension}
The following lemma and its proof follow closely Gromov's Localization Lemma and its proof, see \cite[3.5.B]{Gro89}. 
\begin{lem}\label{lem:ext} Let $U\subset X$  be a dense open subset and $(E,p,s)$  be a local spray  
on $U$  with values in $X$ dominating 
at a point $x\in U$. Then there exists a spray $(\tilde E,\tilde p,\tilde s)$ on $X$ dominating at $x$. 
\end{lem}
\begin{proof}
Shrinking $U$ we may suppose that 
\begin{itemize}
\item $U=X\setminus {\rm supp}(D)$ where $D$ is a reduced effective divisor on $X$, and
\item 
$p\colon E=U\times\AA^r\to U$ is a trivial vector bundle on $U$.
\end{itemize} 
We extend $p\colon E\to U$ to a  trivial vector bundle  $p'\colon E'=X\times\AA^r\to X$ on $X$.
Let 
$\tilde p\colon \tilde E=E'\otimes \OOO_X(-nD)\to X$ where $n\in\NN$. 
Thus, $\tilde E$ is a direct sum of $r$ samples of the line bundle $\OOO_X(-nD)$. We have
\[{\rm Hom}(\OOO_X(-nD),\OOO_X)={\rm Hom}(\OOO_X, \OOO_X(nD)) =H^0(X,\OOO_X(nD)).\]
Pick a section $\sigma\in H^0(X,\OOO_X(D))$ with ${\rm div}(\sigma)=D$. Then $\sigma^n$ 
defines a homomorphism of line bundles
$\OOO_X(-nD)\to\OOO_X$ identical on  $X$ which vanishes to order $n$ on 
$\tilde p^{-1}( {\rm supp}(D))$ 
and restricts to an isomorphism over $U$. The direct sum of these homomorphisms yields
a homomorphism of vector bundles $\psi\colon\tilde E\to E'$
identical on $X$ with similar properties. Extend $s$ to a rational map 
$s'\colon E'\dasharrow X$. Then for all $n\gg 1$ the composition
$\tilde s=s'\circ\psi\colon \tilde E\to X$ is a morphism such that $s|_{\tilde Z} =\tilde p|_{\tilde Z}$ 
where $\tilde Z$ is the zero section of 
$\tilde p\colon \tilde E\to X$. The resulting spray $(\tilde E,\tilde p,\tilde s)$ 
on $X$ is clearly dominating at $x$. 
\end{proof}
The following corollary is immediate. It is implicitly present in the proof of the 
Localization Lemma in \cite[3.5.B]{Gro89}, cf. also \cite[3.5.B$'$]{Gro89} and \cite[Proposition 6.4.2]{For17}.
\begin{cor}\label{cor:loc-ell-subell}
If $X$ is locally elliptic, then it is subelliptic. 
\end{cor}
Next we deduce the following useful fact.
\begin{lem}\label{lem:rank-1}
Any subelliptic  smooth variety $X$ admits a dominating family of sprays of rank $1$. 
\end{lem}
\begin{proof}
Let $(E_i,p_i,s_i)$ be a dominating family of sprays over $X$. For a point $x\in X$ we can find 
a neighborhood $U$ such that $p_i|_U\colon E_i|_U\to U$ is a trivial bundle of rank, say $r_i$. 
Choose a decomposition $E_i|_U=\bigoplus_{j=1}^{r_i} L_{i,j}$ where 
$p_{i,j}=p_i|_{L_{i,j}}\colon L_{i,j}\to U$
are trivial line bundles. Letting $s_{i,j}=s|_{L_{i,j}}\colon L_{i,j}\to X$ yields a family of local rank 1 
sprays $(L_{i,j}, p_{i,j}, s_{i,j})_{j=1,\ldots,r_i}$ on $U$. It is easily seen that this family is dominating at $x$. 
Shrinking $U$ and proceeding in the same way as in the proof of 
Lemma \ref{lem:ext} we extend the latter family to a family of rank 1 sprays 
$(\tilde L_{i,j}, \tilde p_{i,j}, \tilde s_{i,j})_{j=1,\ldots,r_i}$ on the whole
$X$ which is still dominating at $x$. In this way we can construct a  family $(\tilde L_{i,j}, \tilde p_{i,j}, \tilde s_{i,j})_{i,j}$ 
which is dominating at each point of $X$. 
\end{proof}
\subsection{Proof of the main theorem}
In view of Corollary \ref{cor:loc-ell-subell} the following proposition ends 
the proof of Theorem \ref{mthm} from the Introduction. 
\begin{prop}\label{el.t1} 
Let $X$ be a smooth algebraic variety. If $X$ is subelliptic, then it is elliptic.
\end{prop}
\begin{proof}  By Corollary \ref{cor:comp-1-sprays} it suffices to construct
a dominating family of  rank $1$ sprays on $X$. However, Lemma \ref{lem:rank-1} 
provides such a family. 
\end{proof}
\renewcommand{\MR}[1]{}

\end{document}